\documentclass[10pt,a4paper]{amsart}

\usepackage{microtype}
\usepackage{colortbl}
\usepackage{amsthm,amsmath,amssymb}

\newtheorem{theorem}{Theorem}
\newtheorem{lemma}{Lemma}

\theoremstyle{definition}
\newtheorem{definition}[theorem]{Definition}

\begin{document}

\title[Proof of Northshield's conjecture]{Proof of Northshield's conjecture concerning an analogue of Stern's sequence for $\mathbb{Z}[\sqrt{2}]$}


\author{Michael Coons}
\address{School of Mathematical and Physical Sciences\\
University of Newcastle\\
Australia}
\email{Michael.Coons@newcastle.edu.au}

\date{\today}
\thanks{This research was undertaken while visiting the Alfr\'ed R\'enyi Institute of the Hungarian Academy of Sciences; we thank the Institute and its members for their kindness and support.}

\begin{abstract} We prove a conjecture of Northshield by determining the maximal order of his analogue of Stern's sequence for $\mathbb{Z}[\sqrt{2}]$. In particular, if $b$ is Northshield's analogue, we prove that $$\limsup_{n\to\infty}\frac{2b(n)}{(2n)^{\log_3 (\sqrt{2}+1)}}=1.$$
\end{abstract}

\maketitle

\section{Introduction}

{\em Stern's Diatomic sequence} (commonly called {\em Stern's sequence}) is given by $a(0)=0$, $a(1)=1$, and when $n\geqslant 1$ by $$a(2n)=a(n)\quad\mbox{and}\quad a(2n+1)=a(n)+a(n+1).$$ As an analogue of Stern's sequence for the ring $\mathbb{Z}[\sqrt{2}]$, Northshield \cite{N2015} introduced the sequence defined by $b(0)=0$, $b(1)=1$, and in general by 
\begin{align*} b(3n) &= b(n)\\
b(3n+1) &= \sqrt{2}\cdot b(n)+b(n+1)\\
b(3n+2) &= b(n)+\sqrt{2}\cdot b(n+1).
\end{align*}
In joint work with Tyler \cite{CT2014}, answering a question of Berlekamp, Conway, and Guy \cite[page 115]{BCG1982} and improving on a result of Calkin and Wilf \cite{CW1998}, we determined the maximal order of Stern's sequence; in particular, we proved that $\limsup_{n\to\infty}a(n)n^{-\log_2 \varphi}={3^{\log_2 \varphi}}{{5}^{-1/2}},$ where $
\varphi=(1+\sqrt{5})/2$ is the golden mean, and here and throughout this paper, we write $\log_k c$ for the base-$k$ logarithm of the real number $c$. Concerning his analogue, Northshield \mbox{\cite[Cor.~5]{N2015}} showed that \begin{equation}\label{cor:N}\limsup_{n\to\infty}\frac{2b(n)}{(2n)^{\log_3 (\sqrt{2}+1)}}\geqslant 1,\end{equation} and he conjectured that equality holds. 

In this paper, using the method developed by Coons and Tyler \cite{CT2014} (see also Coons and Spiegelhofer \cite{CS2017}), we prove Northshield's conjecture.

\begin{theorem}\label{main} Let $\{b(n)\}_{n\geqslant 0}$ denote Northshield's analogue of Stern sequence as defined above. Then $$\limsup_{n\to\infty}\frac{2b(n)}{(2n)^{\log_3 (\sqrt{2}+1)}}=1.$$
\end{theorem}

This paper is organised as follows. In Section \ref{sec:pre}, we define a piecewise linear function and provide several lemmas comparing it to Northshield's sequence. In Section~\ref{sec:main}, we prove Theorem \ref{main}. Finally, in Section \ref{sec:fm}, we give some further comparisons with Stern's sequence and related values and functions.

\section{Preliminaries}\label{sec:pre}

We proceed along the same lines as the arguments of Coons and Tyler \cite{CT2014} and Coons and Spiegelhofer \cite{CS2017}. In particular, we will define a piecewise linear function $h$, which will serve as an upper bound for the sequence $b$. The benefit in this situation is that $h$ is continuous and (except at a few points) differentiable. As well, the function $h$ will be close to the sequence $b$ for the maximal values of $b$. This closeness will allow us use the asymptotic properties of $h$ to determine the desired asymptotic concerning $b$. 

We start by formally defining the function $h$ and a special sequence of points. 

\begin{definition}\label{hdef} Let $x_n:=3^{n}/2$, $y_n:=(\sqrt{2}+1)^{n}/2$ and let $h:\mathbb{R}_{\geqslant 0}\to \mathbb{R}_{\geqslant 0}$ be the piecewise linear function connecting the set of points $\{(0,0)\}\cup \{(x_n,y_n):n\geqslant 1\}$.
\end{definition}

Northshield proved that\footnote{Our version corrects a small typo in \cite{N2015}.} $$\max\{b(m):m\in[3^{n-1},3^{n}]\}=\frac{(\sqrt{2}+1)^n+(\sqrt{2}-1)^n}{2},$$ and, moreover, the first such maximum in this interval occurs at $m=(3^{n}+1)/2$. The points $\{(x_n,y_n):n\geqslant 1\}$ were chosen to be very close to these maximum values.

\begin{lemma}\label{bound} For $m\geqslant 2$, we have $b(m)\leqslant h(m)+(\sqrt{2}+1)\lfloor \log_3(m)\rfloor$.
\end{lemma}

\begin{proof} In the interval $[x_n,x_{n+1}]$, we have that \begin{align}\nonumber  h(x) &=\frac{h(x_{n+1})-h(x_{n})}{x_{n+1}-x_n}(x-x_n)+h(x_n)\\
\label{hxnline}&=\frac{\sqrt{2}}{2}\left(\frac{\sqrt{2}+1}{3}\right)^n x+(\sqrt{2}+1)^n\left(\frac{2-\sqrt{2}}{4}\right).
\end{align}
Using \eqref{hxnline}, we now can check that the result of the lemma holds in the interval $[x_1,x_2]$; see Table \ref{indhyp} for these values.

\begin{table}[ht]
\caption{Values (showing only three decimal places) demonstrating $a(m)\leqslant h(m)+\lfloor \log_3(m)\rfloor$ for $m=2,\ldots,9$.}
\begin{center}
\begin{tabular}{|c|c|c|c|c|c|c|c|c|} \hline
$m$  & 2 & 3 & 4 & 5 & 6 & 7 & 8 & 9 \\ \hline 
$b(m)$ & $1.414$ & $1$ & $2.828$ & $3$ & $1.414$ & $3$ & $2.828$ & $1$ \\
$h(m)+\lfloor \log_3(m)\rfloor$ & $1.491$ &$3.060$ &$3.629$ &$4.198$ &$4.767$ &$5.336$ &$5.905$ & $7.474$ \\ \hline
\end{tabular}
\end{center}
\label{indhyp}
\end{table}%

We will proceed by induction. Suppose that the result holds in $[3^{n-1},3^n]$ and consider $[3^n,3^{n+1}]$. As mentioned above, the first occurring maximum value of $b$ in $[3^n,3^{n+1}]$ is $$b\left(\frac{3^{n+1}+1}{2}\right)=\frac{(\sqrt{2}+1)^{n+1}+(\sqrt{2}-1)^{n+1}}{2}.$$ As $(3^{n+1}+1)/2\in[x_{n+1},x_{n+2}]$, by \eqref{hxnline}, at this value we have 
\begin{align} 
\nonumber h\left(\frac{3^{n+1}+1}{2}\right) +\left\lfloor \log_3\left(\frac{3^{n+1}+1}{2}\right)\right\rfloor &= \frac{\sqrt{2}}{2}\left(\frac{\sqrt{2}+1}{3}\right)^{n+1} \left(\frac{3^{n+1}+1}{2}\right)\\
\nonumber &\qquad\quad+(\sqrt{2}+1)^{n+1}\left(\frac{2-\sqrt{2}}{4}\right)+n+1 \\
\label{hmn}&=\left( \frac{\sqrt{2}}{4\cdot 3^{n+1}}+\frac{1}{2}\right)(\sqrt{2}+1)^{n+1}+n+1\\
\nonumber &>\frac{(\sqrt{2}+1)^{n+1}+(\sqrt{2}-1)^{n+1}}{2}\\
\nonumber &=b\left(\frac{3^{n+1}+1}{2}\right), 
\end{align} so the lemma holds for the value $3^{n+1}+1/2$. 

Now if $m\in[(3^{n+1}+1)/2, 3^{n+1}]$, since the lemma holds for the value $(3^{n+1}+1)/2$ and $b$ takes it's maximal value in $[3^{n},3^{n+1}]$ at $3^{n+1}+1/2$, we have $$b(m)\leqslant b\left(\frac{3^{n+1}+1}{2}\right)\leqslant h\left(\frac{3^{n+1}+1}{2}\right) +\left\lfloor \log_3\left(\frac{3^{n+1}+1}{2}\right)\right\rfloor\leqslant h(m)+\lfloor \log_3(m)\rfloor,$$ where the last inequality follows from the fact that $h$ is monotonically increasing. Thus the lemma holds in the interval $[(3^{n+1}+1)/2, 3^{n+1}]$. It remains to show that the result holds for $m\in[3^{n},(3^{n+1}-1)/2]$.

If $m=3k\in[3^{n},(3^{n+1}-1)/2]$, then $k\in[3^{n-1},3^n]$. By Northshield's definition and the induction hypothesis, we have $$b(m)=b(3k)=b(k)\leqslant h(k)+\lfloor \log_3(k)\rfloor\leqslant h(m)+\lfloor \log_3(m)\rfloor,$$ where as above, the last inequality follows from the monotonicity of $h$.

If $m=3k+1\in[3^{n},(3^{n+1}-1)/2]$, then $k+1\in[3^{n-1},(3^{n}+1)/2]$. Note that in this case, using \eqref{hxnline}, we have \begin{multline}\label{>1}
h(3k+1)-(\sqrt{2}+1)h(k+1)\\ =\frac{\sqrt{2}}{2}\left(\frac{\sqrt{2}+1}{3}\right)^{n+1} (3k+1)-\frac{3\sqrt{2}}{2}\left(\frac{\sqrt{2}+1}{3}\right)^{n+1} (k+1)\\ 
=-\sqrt{2}\left(\frac{\sqrt{2}+1}{3}\right)^{n+1}\in(0,1).
\end{multline}
Now \begin{align*}
b(m)=b(3k+1)&=\sqrt{2}\cdot b(k)+b(k+1)\\
&\leqslant (\sqrt{2}+1)\cdot\max\{b(k),b(k+1)\}\\
&\leqslant (\sqrt{2}+1)\left(h(k+1)+\lfloor \log_3(k+1)\rfloor\right),
\end{align*} again appealing to the monotonicity of $h$. Combining this with \eqref{>1} and using the induction hypothesis, we have \begin{align*} b(m)=b(3k+1)& \leqslant h(3k+1)+(\sqrt{2}+1)\lfloor \log_3(k+1)\rfloor+1\\ &\leqslant h(3k+1)+(\sqrt{2}+1)\lfloor \log_3(3k+1)\rfloor,\end{align*} since here $\lfloor \log_3(k+1)\rfloor=n$ and $\lfloor \log_3(3k+1)\rfloor=n+1$. Thus the result holds for $m=3k+1\in[3^{n},(3^{n+1}-1)/2]$.

The remaining case is $m=3k+2\in[3^{n},(3^{n+1}-1)/2]$. But this follows easily from the monotonicity of $h$, as again we have $$b(m)=b(3k+2)=b(k)+\sqrt{2}\cdot b(k+1)\leqslant (\sqrt{2}+1)\cdot\max\{b(k),b(k+1)\}.$$ Thus the previous case along with the monotonicity of $h$ gives \begin{align*} b(m)=b(3k+2)& \leqslant h(3k+1)+(\sqrt{2}+1)\lfloor \log_3(3k+1)\rfloor\\ &\leqslant h(3k+2)+(\sqrt{2}+1)\lfloor \log_3(3k+2)\rfloor.\end{align*} This finishes the proof of the lemma.
\end{proof}

\begin{lemma}\label{limeq} We have $$\limsup_{m\to\infty}\frac{b(m)}{h(m)}=1.$$
\end{lemma}

\begin{proof} Set $m_n:=(3^{n+1}+1)/2$. Note that $b(m_n)\sim (\sqrt{2}+1)^{n+1}/2$ and also, recalling \eqref{hmn}, that $$h\left(\frac{3^{n+1}+1}{2}\right)=\left( \frac{\sqrt{2}}{4\cdot 3^{n+1}}+\frac{1}{2}\right)(\sqrt{2}+1)^{n+1}\sim \frac{(\sqrt{2}+1)^{n+1}}{2}.$$ Thus $$1=\lim_{n\to\infty}\frac{b(m_n)}{h(m_n)}\leqslant \limsup_{m\to\infty} \frac{b(m)}{h(m)}\leqslant \limsup_{m\to\infty} \frac{h(m)+(\sqrt{2}+1)\lfloor \log_3 m\rfloor}{h(m)}=1,$$ where the last inequality is given by Lemma \ref{bound} and the final equality follows since for $m\in[x_n,x_{n+1}]$, we have \begin{equation*} \frac{(\sqrt{2}+1)\lfloor \log_3 m\rfloor}{h(m)}\geqslant \frac{2\lfloor \log_3 x_n\rfloor}{h(x_{n+1})}\geqslant \frac{2n}{(\sqrt{2}+1)^{n+1}}.\qedhere\end{equation*}
\end{proof}

\begin{lemma}\label{xbound} For $x>3/2$, we have $2\cdot h(x)\leqslant (2x)^{\log_3(\sqrt{2}+1)}.$
\end{lemma}

\begin{proof} Firstly, note that for the sequence $x_n$ as given in Definition \ref{hdef} and $n\geqslant 1$, we have $\log_3 x_n=n-\log_3 2$, so that $$2\cdot h(x_n)=2\cdot y_n=(\sqrt{2}+1)^n=(\sqrt{2}+1)^{\log_3x_n+\log_3 2}=(2x_n)^{\log_3(\sqrt{2}+1)},$$ which shows the lemma holds for the values $x_n$.

Write $$H(x):=2\cdot h(x)- (2x)^{\log_3(\sqrt{2}+1)}.$$ If $H(x)>0$ for some $x\in[x_n,x_{n+1}]$, then since $H$ is differentiable in $(x_n,x_{n+1})$ there is some $w\in(x_n,x_{n+1})$ where $H$ attains a maximum value. But \begin{align*}\frac{{\rm d}^2}{{\rm d}x^2}H(x)&=\frac{{\rm d}^2}{{\rm d}x^2}\left\{- (2x)^{\log_3(\sqrt{2}+1)}\right\}\\
&=-2^{\log_3(\sqrt{2}+1)}\log_3(\sqrt{2}+1)(\log_3(\sqrt{2}+1)-1)x^{\log_3(\sqrt{2}+1)-2},\end{align*} which is positive for all $x\in[x_n,x_{n+1}]$. Thus $H(x)\leqslant 0$ for all $x>x_1=3/2$ proving the lemma.
\end{proof}

\section{Proof of Northshield's conjecture}\label{sec:main}

\begin{proof}[Proof of Theorem \ref{main}] By Lemmas \ref{limeq} and \ref{xbound} we have $$1\leqslant \limsup_{m\to\infty} \frac{2b(m)}{(2m)^{\log_3(\sqrt{2}+1)}}\leqslant \limsup_{m\to\infty}\frac{b(m)}{h(m)}=1,$$ where the first inequality, recorded in \eqref{cor:N}, is due to Northshield.
\end{proof}

\section{Further remarks}\label{sec:fm}

Both Stern's sequence and Northshield's analogue are examples of $k$-regular sequences as defined by Allouche and Shallit in their seminal paper \cite{AS1992}. For an integer $k\geqslant 2$, an integer-valued sequence $f$ is called {\em $k$-regular} provided there exist a positive integer $d$, a finite set of matrices $\mathcal{M}=\{{\bf M}_0,\ldots,{\bf M}_{k-1}\}\subseteq \mathbb{Z}^{d\times d}$, and vectors ${\bf v},{\bf w}\in \mathbb{Z}^d$ such that $$f(n)={\bf w}^T {\bf M}_{w}{\bf v},$$ where ${\bf M}_{w}={\bf M}_{i_0}\cdots{\bf M}_{i_s}$ and $w={i_0}\cdots {i_s}$ is the reversal of the base-$k$ expansion $(n)_k={i_s}\cdots {i_0}$; see \cite[Lemma~4.1]{AS1992}. We call the tuple $({\bf w}, \mathcal{M}, {\bf v})$ the {\em linear representation} of the $k$-regular sequence $f$.

Stern's sequence $a$ is $2$-regular and has linear representation $$\left([1\ 0],\left\{{\bf A}_0,{\bf A}_1\right\}=\left\{\left[\begin{array}{rr} 1&1\\ 0&1\end{array}\right],\left[\begin{array}{rr} 1&0\\ 1&1\end{array}\right]\right\},[1\ 0]\right),$$ whereas Northshield's sequence $b$ is $3$-regular and has linear representation $$\left([1\ 0],\{{\bf B}_0,{\bf B}_1,{\bf B}_2\}=\left\{\left(\begin{matrix} 1 & 0\\ \sqrt{2} & 1 \end{matrix}\right),\left(\begin{matrix} \sqrt{2} & 1\\ 1 & \sqrt{2} \end{matrix}\right),\left(\begin{matrix} 1& \sqrt{2} \\ 0 & 1 \end{matrix}\right)\right\},[0\ 1]\right).$$ This representation of $k$-regular sequences looks a lot like the matrix version of a linear recurrence (coefficients of rational power series), and indeed, $k$-regular sequences are sometimes known as `radix-rational' sequences. 

The method used here can give analogous results for other $k$-regular sequences. Essentially this can be done using the following recipe for a $k$-regular sequence $f$:
\begin{itemize}
\item[1.] Determine the maximal values of $f$ between consecutive powers of $k$ and where they first occur.
\item[2.] Find a piecewise linear function $h$ that is both monotonically increasing and close enough to the above determined maximal values of $f$ so that one has $\limsup_{n\to\infty}f(n)/h(n)=1.$
\item[3.] Show that the desired maximal order holds for $h$ and deduce from Step 2 that it also holds for $f$.
\end{itemize}
Compared to Step 1, in general, Steps 2 and 3 should be relatively easy. The difficulty in Step 1 is related to questions surrounding the joint spectral radius of finite sets of (in this case) integer matrices. 

The {\em joint spectral radius} of a finite set of matrices $\mathcal{M}=\{{\bf M}_0,{\bf M}_1,\ldots, {\bf M}_{k-1}\}$, denoted $\rho(\mathcal{M})$, is defined as the real number $$\rho(\mathcal{M})=\limsup_{n\to\infty}\max_{0\leqslant i_0,i_1,\ldots,i_{n-1}\leqslant k-1}\left\| {\bf M}_{i_0}{\bf M}_{i_1}\cdots{\bf M}_{i_{n-1}}\right\|^{1/n},$$ where $\|\cdot\|$ is any (submultiplicative) matrix norm. It is quite clear that when all of the ${\bf M}_i$ are equal, say to a matrix ${\bf M}$, the joint spectral radius of $\mathcal{M}$ is equal to the spectral radius of ${\bf M}$. The joint spectral radius was introduced by Rota and Strang \cite{RS1960} and has a wide range of applications. For an extensive treatment, see Jungers's monograph \cite{J2009}. 

For the examples of Stern's and Northshield's sequences, the joint spectral radii are the golden and silver means, respectively. That is, $$\rho\left(\left\{{\bf A}_0,{\bf A}_1\right\}\right)=\varphi\quad\mbox{and}\quad \rho\left(\left\{{\bf B}_0,{\bf B}_1,{\bf B}_1\right\}\right)=\rho\left({\bf B}_1\right)=\sqrt{2}+1.$$ The result for the Stern sequence has been known for some decades already, and it seems that for Northshield's sequence, Theorem \ref{main} provides proof; see Coons \cite{IJFCS} for additional details. 

If one can find the joint spectral radius of the set $\mathcal{M}$ associated to $f$, then one can probably find the maximal values of $f$, though in practice, this seems to have been done in the other direction within the research of this area. 

Where these maximal values occur is related to a very interesting and very open question due to Lagarias and Wang \cite{LW1995}. The finite set of integer matrices $\mathcal{M}$ is said to satisfy the {\em finiteness property} provided there is a specific finite product ${\bf M}_{i_0}\cdots {\bf M}_{i_{m-1}}$ of matrices from $\mathcal{M}$ such that $\rho({\bf M}_{i_0}\cdots {\bf M}_{i_{m-1}})^{1/m}=\rho(\mathcal{M}).$ Currently, there is no general way to determine if such a set $\mathcal{M}$ satisfies the finiteness property. 

In the cases of Stern's and Northshield's sequences,  both sets of matrices satisify the finiteness property. For Stern's sequence, the finite product is ${\bf A}_0{\bf A}_1$ and for Northshield's sequence it is the single matrix ${\bf B}_{1}$.

\bibliographystyle{amsplain}
\providecommand{\bysame}{\leavevmode\hbox to3em{\hrulefill}\thinspace}
\providecommand{\MR}{\relax\ifhmode\unskip\space\fi MR }
\providecommand{\MRhref}[2]{%
  \href{http://www.ams.org/mathscinet-getitem?mr=#1}{#2}
}
\providecommand{\href}[2]{#2}


\end{document}